\documentclass[11pt,reqno]{amsart}
\usepackage{tikz}
\usepackage{subfig}
\usepackage{epstopdf}
\usepackage{epsfig}
\usepackage{math}
\usepackage{tikz}
\graphicspath{{./figures/}}
\usetikzlibrary{shapes,arrows}
\tikzstyle{decision} = [diamond, draw, fill=blue!20, 
    text width=4.5em, text badly centered, node distance=3cm, inner sep=0pt]
\tikzstyle{block} = [rectangle, draw, fill=blue!20, 
    text width=5em, text centered, rounded corners, minimum height=4em]
\tikzstyle{line} = [draw, -latex']
\tikzstyle{cloud} = [draw, ellipse,fill=red!20, node distance=3cm,
    minimum height=2em]

\usetikzlibrary{positioning}
\tikzset{main node/.style={circle,fill=blue!20,draw,minimum size=1cm,inner sep=0pt},  }

\begin{document}
\title[Entropy dissipation on graphs]{Entropy dissipation of Fokker-Planck equations on graphs}
\author[Chow]{Shui-Nee Chow}
\address{School of Mathematics, Georgia institute of technology, Atlanta.}
\author[Li]{Wuchen Li}
\address{Department of Mathematics, University of California, Los Angeles.}
\author[Zhou]{Haomin Zhou}
\email{chow@math.gatech.edu }
\email{wcli@math.ucla.edu}
\email{hmzhou@math.gatech.edu}
\thanks{This work is partially supported by NSF Awards DMS-1419027, DMS-1620345, and ONR Award N000141310408.}
\maketitle
\begin{abstract}
We study the nonlinear Fokker-Planck equation on graphs, which is the gradient flow in the space of 
probability measures supported on the nodes with respect to the discrete Wasserstein metric. The energy functional driving the gradient flow consists of a Boltzmann entropy, a linear potential and a quadratic interaction energy. We show that the solution converges to the Gibbs measures exponentially fast with a rate that can be given analytically. The continuous analog of this asymptotic rate is related to the Yano's formula. \end{abstract}

\section{Introduction}
Optimal transport theory reveals many deep connections between partial differential equations and geometry. For example, in the seminal work \cite{JKO}, it is proved that the linear Fokker-Planck equation (FPE) is the gradient flow 
of a free energy in the probability space equipped with Wasserstein metric \cite{G1, otto2001, vil2003, vil2008}. This gradient flow interpretation has been extended to mean field settings, in which 
the free energy contains an interaction energy \cite{am2006}. Many studies have been carried out showing that the 
solution of FPE converges to its equilibrium in an exponential rate, and this is known as the entropy dissipation
in the literature \cite{carrillo2003kinetic,d2001, MC2000}.

The goal of this paper is studying the entropy dissipation of FPE in discrete settings, for example on finite graphs. Such a consideration is motivated by applications in biology, game theory, and  numerical schemes for partial differential equations (PDEs). The optimal transport metric on graphs has been established  by several groups independently \cite{chow2012, maas2011gradient, Mielke2011}. The gradient flow structure based such a metric 
attracts a lot of attentions in recent years. For example, Mass and Erbar studied the discrete heat flow, and further gave the Ricci curvature lower bound in the discrete space \cite{erbar2012ricci}. More generalizations are followed in \cite{Maas1, Maas3, Maas2}. Mielke proposed the discrete reaction diffusion equation \cite{mielke2013geodesic}. Erbar, Fathi and collaborators introduced a discrete McKean-Vlasov equation \cite{G6}, 
which is the evolution equation for the probability density function of mean field Markov process. Various convergence properties of these gradient flows have been brought into attentions as well \cite{T2016, erbar2012ricci, Maas3}.

Following the setups in \cite{chow2012}, we further study the dynamical properties of the gradient flows in the discrete Wasserstein geometry in this paper. Special attention are given to a free energy containing a quadratic interaction energy, a linear potential and the Boltzmann entropy. In this case, the gradient flow can be viewed as the nonlinear FPE on graphs, which is a set of ordinary differential equations (ODEs). We show that the solution of FPE 
converges to, the unique or one of the multiple when the free energy is non-convex, Gibbs measure 
exponentially fast,  which mimics the entropy dissipation property, but in a discrete space. 
We further provide an explicit formula that bounds the convergence rate. The continuous analog of the asymptotic rate formula is related to the Yano's formula in Riemanian geometry \cite{yano1, yano2}.   

The structure of this paper is arranged as follows. We review discrete 2-Wasserstein 
metric and Fokker-Planck equations on graphs in the next section, then study its convergence 
in Section \ref{dissipation}. In Section \ref{analog}, we discuss some properties stemed from the 
convergence rate, including the connection with Yano's formula. 

\section{Optimal transport on finite graphs} \label{review}
In this section, we briefly review the constructions of 2-Wasserstein metric and corresponding FPE on a graph . We mainly follow the approaches given in \cite{chow2012, maas2011gradient}, with some modified notations for a simpler 
presentation.

Consider a weighted finite graph $G=(V, E, \omega)$, where 
$V=\{1,2,\cdots, n\}$
is the vertex set, $E$ is the edge set, and $\omega=(\omega_{ij})_{i,j\in V}$ is the weight of each edge,
\begin{equation*}
\omega_{ij}=\begin{cases}\omega_{ji}>0
& \textrm{if $(i,j)\in E$}\\
0 & \textrm{otherwise}
\end{cases}\ .
\end{equation*}
We assume that $G$ is undirected and contains no self loops
or multiple edges. The adjacency set of vertex $i\in V$ is denoted by
\begin{equation*}
N(i)=\{j\in V\mid (i,j)\in E\}\ .
\end{equation*} 
The probability set (simplex) supported on all vertices of $G$ is defined by 
\begin{equation*}
\mathcal{P}(G)=\{( \rho_i)_{i=1}^n \in \mathbb{R}^{n}\mid \sum_{i=1}^n \rho_i=1\ ,\quad  \rho_i\geq 0\ , \quad \textrm{for any $i\in V$}\}\ ,
\end{equation*}
where $ \rho_i$ is the discrete probability function at node $i$.  Its interior is denoted by
 $\mathcal{P}_o(G)$. We introduce the the following notations and operations 
 on $G$ and $\mathcal{P}(G)$ and use them for the construction the discrete 2-Wasserstein metric.

A {\em vector field} $v=(v_{ij})_{i,j\in V}\in \mathbb{R}^{n\times n}$ on $G$ is a  
{\em skew-symmetric matrix} on the edges set $E$:  
\begin{equation*}
v_{ij}=\begin{cases}-v_{ij} & \textrm{if $(i,j)\in E$}\\
0 & \textrm{otherwise}
\end{cases}\ .
\end{equation*}
Given a function $\Phi=(\Phi_i)_{i=1}^n\in \mathbb{R}^{n}$ defined on the nodes of $G$, a
{\em potential vector field} $\nabla_G\Phi=(\nabla_G\Phi_{ij})_{i,j\in V}\in \mathbb{R}^{n\times n}$ refers to
\begin{equation*}
\nabla_G\Phi_{ij}=\begin{cases}\sqrt{\omega_{ij}}(\Phi_i-\Phi_j) & \textrm{if $(i,j)\in E$}\\
0 & \textrm{otherwise}
\end{cases}\ .
\end{equation*}
For a given probability function $\rho\in \mathcal{P}(G)$ and a vector field $v$,  we define 
the product $\rho v\in \mathbb{R}^{n\times n}$, called {\em flux function} on $G$, by 
\begin{equation*}
\rho v:=(v_{ij} \theta_{ij}(\rho))_{(i,j)\in E}\ ,
\end{equation*}
where $\theta_{ij}(\rho)$ are specially chosen functions. For example, $\theta_{ij}(\rho)$ can be the logarithmic mean or an upwind
function of $\rho$, which are used in \cite{chow2012, maas2011gradient}, see more details in \cite{li-thesis}. In this paper, we select $\theta_{ij}(\rho)$ as the average of $ \rho_i$ and $ \rho_j$, i.e. 
 \begin{equation*}
 \theta_{ij}(\rho)=\frac{\rho_i+\rho_j}{2}\ ,\quad \textrm{for any $(i,j)\in E$}\ ,
\end{equation*}
for the simplicity of illustration. 

We define the {\em divergence} of $\rho v$ on $G$ by
\begin{equation*}
\textrm{div}_G(\rho v):=-\biggl(\sum_{j\in N(i)}\sqrt{\omega_{ij}}v_{ij}\theta_{ij}(\rho)\biggr)_{i=1}^n\in \mathbb{R}^{n}\ .
\end{equation*}
Given two vector fields $v$, $w$ on a graph and $\rho \in \mathcal{P}(G)$.
The discrete {\em inner product} is defined by
\begin{equation*}
(v, w)_ \rho:=\frac{1}{2}\sum_{(i,j)\in E} v_{ij}w_{ij}\theta_{ij}(\rho)\ . 
\end{equation*}
The coefficient $1/2$ in front of the summation accounts for the fact that every edge in $G$ is counted 
twice. In particular, we have 
\begin{equation*}
(v, v)_ \rho=\frac{1}{2}\sum_{(i,j)\in E} v_{ij}^2\theta_{ij}(\rho)\ .
\end{equation*}

With these definitions, we introduce an integration by part formula on graphs that will be used throughout 
this paper: {\it For any vector field $v$ and potential function $\Phi$ on a graph, the following properties hold
$$
-\sum_{i=1}^n \textrm{div}_G(\rho v)|_i \Phi_i=(v, \nabla_G\Phi)_ \rho\ ,
$$
and 
$$
\sum_{i=1}^n \textrm{div}_G(\rho v)|_i =0\ .
$$
}
\begin{proof} From $v_{ij} + v_{ji} = 0$,  we write
\begin{equation}
\begin{split}
-\sum_{i=1}^n \textrm{div}_G(\rho v)|_i \Phi_i
=&  \sum_{i=1}^n \sum_{j\in N(i)} \sqrt{\omega_{ij}} v_{ij}\theta_{ij}\Phi_i \nonumber \\
=&  \frac{1}{2}(\sum_{(i,j)\in E}\sqrt{\omega_{ij}} v_{ij}\Phi_i\theta_{ij}+ \sum_{(j,i)\in E}\sqrt{\omega_{ij}} v_{ji}\Phi_j \theta_{ji})\nonumber\\
=&\frac{1}{2}\sum_{(i,j)\in E} v_{ij} \sqrt{\omega_{ij}} (\Phi_i-\Phi_j)\theta_{ij}\label{eq:v_{ij}+v_{ji}=0}  \\
=&(\nabla_G\Phi, v)_\rho\ .\nonumber
\end{split}
\end{equation}
Let $\Phi=(1,\cdots, 1)^T$, then $\sum_{i=1}^n \textrm{div}_G(\rho v)|_i =-\sum_{i=1}^n (v, \nabla_G1)_ \rho=0$ .

\end{proof}
It is worth to remark that we prefer not to replace $\theta_{ij}$ by its explicit formula, as done in \cite{chow2012, 
maas2011gradient}, to emphasize the freedom of using different $\theta_{ij}$, which can result in different
definitions for the flux function, divergence operator and inner product, and hence lead to
different formulas for the discrete $2$-Wasserstein metric. 
%
\subsection{2-Wasserstein metric on a graph}
The discrete analogue of $2$-Wasserstein metric $\mathcal{W}_2$ on probability set $\mathcal{P}_o(G)$
can be given as following.
{\em For any given $ \rho^0$, $ \rho^1\in \mathcal{P}_o(G)$, define 
\begin{equation}\label{metric}
\mathcal{W}_{2}^2( \rho^0, \rho^1):=\inf_{v}~\{\int_0^1(v, v)_\rho dt~:~
\frac{d\rho}{dt}+\textrm{div}_G(\rho v)=0\ ,\quad \rho(0)= \rho^0\ ,\quad  \rho(1)= \rho^1\}\ ,
\end{equation}
where the infimum is taken over all vector fields $v$ on a graph, and $\rho$ is a 
continuously differentiable curve
$ \rho\colon [0,1]\rightarrow \mathcal{P}_o(G)$. }

This is the corresponding Benamou-Brenier formula \cite{bb} in discrete space. Modifying a similar proof as 
given in \cite{maas2011gradient}, one can show the following lemma, see details in \cite{li-thesis}.
\begin{lemma}\label{Hodge}
Given a vector field on a graph $v=(v_{ij})_{(i,j)\in E}$ with $v_{ij}=-v_{ji}$, and a measure $\rho\in\mathcal{P}_o(G)$, there exists a unique decomposition, such that 
\begin{equation*}
v=\nabla_G\Phi+u\ , \quad \textrm{and}\quad \textrm{div}_G(\rho u)=0\ ,
\end{equation*}
where $\Phi$ is a function defined on $V$. In addition, the following property holds,
\begin{equation*}
(v, v)_\rho=(\nabla_G\Phi, \nabla_G\Phi)_\rho+(u, u)_\rho\ .
\end{equation*}
\end{lemma}
 One may view Lemma \ref{Hodge} as a discrete analogue of the well-known Hodge decomposition. Using it, the metric \eqref{metric} can be proven equivalent to 
\begin{equation} \label{w2_1}
\begin{split}
\left(\mathcal{W}_{2}( \rho^0, \rho^1)\right)^2=&\inf_{\Phi}\{\int_0^1(\nabla_G\Phi, \nabla_G\Phi)_\rho dt ~:~\frac{d\rho}{dt}+\textrm{div}_G( \rho \nabla_G\Phi)=0\ ,~ \rho(0)= \rho^0\ ,~ \rho(1)= \rho^1\}\ ,
\end{split}
\end{equation}
where the infimum is taken over all potentials $\Phi\colon[0,1]\rightarrow\mathbb{R}^n$.
 
Let us denote the tangent space at $\rho \in \mathcal{P}_o(G)$ as  
\begin{equation*}
T_\rho\mathcal{P}_o(G)=\{(\sigma_i)_{i=1}^n\in \mathbb{R}^n\mid \sum_{i=1}^n\sigma_i=0\}\ .
\end{equation*}
We define a weighted graph Laplacian matrix $L(\rho)\in \mathbb{R}^{n\times n}$:
$$ L(\rho)=-D^T \Theta(\rho) D\ ,$$
where $D \in \mathbb{R}^{|E|\times |V|}$ is the discrete gradient matrix 
\begin{equation*} 
D_{(i,j)\in E, k\in V}=\begin{cases}
\sqrt{\omega_{ij}} & \textrm{if $i=k\ ;$}\\ 
-\sqrt{\omega_{ij}} & \textrm{if $j=k\ ;$}\\
0 & \textrm{otherwise}\ ;
\end{cases}
\end{equation*}
and
$\Theta\in \mathbb{R}^{|E|\times |E|}$ is the diagonal weighted matrix
\begin{equation*}
\Theta_{(i,j)\in E, (k,l)\in E}=\begin{cases}
\theta_{ij}(\rho) & \textrm{if $(i,j)=(k,l)\in E\ ;$}\\ 
0 & \textrm{otherwise}\ .
\end{cases}
\end{equation*}
We would like to emphasize that the weights in $L(\rho)$ 
depend on the distribution $\rho$, and this is very different from the commonly used graph Laplacian matrices. 

\begin{lemma}\label{lem2}
For any given $\sigma \in T_\rho\mathcal{P}_o(G)$, 
there exists a unique function $\Phi$, up to a constant shift, satisfying $$\sigma =L(\rho)\Phi= -\textrm{div}_G(\rho\nabla_G \Phi)\ .$$ 
\end{lemma}
\begin{proof}
If $\rho\in \mathcal{P}_o(G)$, all diagonal entries of the weighted matrix $\Theta(\rho)$ is nonzero. 
Consider
$$\Phi^TL(\rho)\Phi=\frac{1}{2}\sum_{(i,j)\in E}(\Phi_i-\Phi_j)^2\theta_{ij}(\rho)=0\ ,$$ 
then $\Phi_i=\Phi_j$ for any $(i,j)\in E$. Since $G$ is connected, $\Phi_i=\textrm{constant}$ for any $i\in V$. Thus $0$ is a simple eigenvalue of $L(\rho)$ and $ L(\rho)(1,\cdots, 1)^T=0$, i.e. $(1,\cdots, 1)^T\in \textrm{ker}(L(\rho))$. 
Hence
$\textrm{dim}(\mathbb{R}^n/{\textrm{ker}(L(\rho))})=\textrm{dim}(\textrm{Ran}(L(\rho)))=\textrm{dim}(\mathcal{T}_\rho\mathcal{P}_o(G))=n-1$.
Since $\sum_{i=1}^n\textrm{div}_G(\rho\nabla_G\Phi)_i=0$, we have $\textrm{Ran}(L(\rho))\subset \mathcal{T}_\rho\mathcal{P}_o(G)$. Therefore $$(\mathbb{R}^n/{\textrm{ker}(L(\rho))})\cong\textrm{Ran}(L(\rho))=\mathcal{T}_\rho\mathcal{P}_o(G)\ ,$$ which proves the lemma.
\end{proof}

Based on Lemma \ref{lem2}, we write 
$$L(\rho)=T\begin{pmatrix}
0 & & &\\
& \lambda_{sec}(L(\rho))& &\\
& & \ddots & \\
& & & {\lambda_{\max}(L(\rho))}
\end{pmatrix}T^{-1} \ ,$$
where $0<\lambda_{sec}(L(\rho))\leq\cdots\leq \lambda_{\max}(L(\rho))$ are eigenvalues of $L(\rho)$ arranged in ascending order, and $T$ is its corresponding eigenvector matrix. 
We denote the pseudo-inverse of $L(\rho)$ by
$$L^{-1}(\rho)=T\begin{pmatrix}
0 & & &\\
& \frac{1}{\lambda_{sec}L(\rho)}& &\\
& & \ddots & \\
& & & \frac{1}{\lambda_{max}L(\rho)}
\end{pmatrix}T^{-1} \ .$$
Then matrix $L^{-1}(\rho)$ endows an inner product on $T_\rho\mathcal{P}_o(G)$. 
\begin{definition}\label{d9}
For any two tangent vectors $\sigma^1,\sigma^2\in T_\rho\mathcal{P}_o(G)$, 
define the inner product 
$g :T_\rho\mathcal{P}_o(G)\times T_\rho\mathcal{P}_o(G)  \rightarrow \mathbb{R}$ by
\begin{equation*}\begin{split}\label{formula}
g(\sigma^1,\sigma^2):=(\Phi^1)^TL(\rho)(\Phi^2)=(\sigma^1)^T L^{-1}(\rho) \sigma^2\ ,\end{split} 
\end{equation*}
where $\sigma^1=L(\rho)\Phi^1$ and $\sigma^2=L(\rho)\Phi^2$. 
\end{definition} 
Hence metric \eqref{metric} is equivalent to 
 \begin{equation}\label{new}
\left(\mathcal{W}_2( \rho^0,  \rho^1)\right)^2=\inf\{\int_0^1 \dot\rho^T L^{-1}(\rho)\dot \rho ~dt~:~ 
\quad  \rho(0)= \rho^0\ ,\quad  \rho(1)= \rho^1\ , \quad  \rho\in\mathcal{C}\}\ ,
\end{equation} 
where $\mathcal{C}$ is the set of all continuously differentiable curves $ \rho(t): [0,1]\rightarrow \mathcal{P}_o(G)$.  From \eqref{new}, it is clear that $(\mathcal{P}_o(G), \mathcal{W}_2)$ is a finite dimensional Riemannian manifold.

\subsection{Gradient flows on finite graphs}
We now consider the gradient flow of $\mathcal{F}\colon \mathcal{P}(G)\rightarrow \mathbb{R}$, $\mathcal{F}\in C^2$, on the Riemannian manifold ($\mathcal{P}_o(G), \mathcal{W}_2$). 
 \begin{theorem}[Gradient flows]\label{Derive}
For a finite graph $G$ and a constant $\beta> 0$, the gradient flow of $\mathcal{F}(\rho)$ on $(\mathcal{P}_o(G), \mathcal{W}_{2})$ is 
$$\frac{d\rho}{dt}=-L(\rho)\nabla_\rho\mathcal{F}(\rho)\ ,$$
i.e.
\begin{equation} \label{a1}
\begin{split}
\frac{ d \rho_i}{dt}=\sum_{j\in N(i)}\omega_{ij}\theta_{ij}(\rho)\big( \frac{\partial}{\partial  \rho_j}\mathcal{F}(\rho) -\frac{\partial }{\partial  \rho_i} \mathcal{F}(\rho)\big)\ . 
\end{split}
\end{equation}
\end{theorem}
\begin{proof}[Proof]
For any $\sigma\in T_\rho\mathcal{P}_o(G)$, there exists $\Phi\in \mathbb{R}^n$, such that 
$\sigma=-\textrm{div}_G(\rho\nabla_G\Phi)=L(\rho)\Phi$. By Definition \ref{formula},   
\begin{equation}\label{d1}
\begin{split}
(\frac{d\rho}{dt}, \sigma)_\rho=&=\frac{d\rho}{dt}^TL^{-1}(\rho)\sigma=\sum_{i=1}^n \frac{d\rho_i}{dt}\Phi_i\ .
\end{split}
\end{equation}
On the right hand side, 
\begin{equation}\label{new1}
\begin{split}
\textrm{d}\mathcal{F}(\rho)\cdot{\sigma}
=&\sum_{i=1}^n\frac{\partial}{\partial  \rho_i}\mathcal{F}(\rho)\cdot \sigma_i= F(\rho)^TL(\rho)\Phi \\
=&\Phi^TL(\rho)F(\rho)=-\sum_{i=1}^n \Phi_i \textrm{div}_G(\rho\nabla_GF(\rho))_i\ , \\
\end{split}
\end{equation}
where we denote $F(\rho)=(F_i(\rho))_{i=1}^n=(\frac{\partial }{\partial  \rho_i}\mathcal{F}(\rho))_{i=1}^n$.
Combining \eqref{d1}, \eqref{new1}, and the definition of gradient flow on manifold, we obtain  
\begin{equation*}
\begin{split}
0=&(\frac{d\rho}{dt}, \sigma)_{\rho}+\textrm{d}\mathcal{F}(\rho)\cdot{\sigma} \\
=&\sum_{i=1}^n \{\frac{d \rho_i}{dt}-\textrm{div}_G(\rho\nabla_G F(\rho)) \}\Phi_i\ .
\end{split}
\end{equation*}
Since $(\Phi_i)_{i=1}^n\in \mathbb{R}^n$ is arbitrary, we must have   
\begin{equation*}
\frac{d \rho_i}{dt}+\sum_{j\in N(i)} \omega_{ij} \theta_{ij}(\rho)\big(F_i(\rho)-F_j( p)\big)=0
\end{equation*}
for all $i\in V$, which is \eqref{a1}.
\end{proof}
Clearly,  \eqref{a1} is the discrete analog of Wasserstein gradient flow in continuous space
\begin{equation*}
\frac{\partial \rho}{\partial t}=\nabla\cdot (\rho\nabla \frac{\delta}{\delta \rho}\mathcal{F}(\rho))\ ,
\end{equation*}
where $\frac{\delta}{\delta \rho}\mathcal{F}$ is the first variation of $\mathcal{F}$. In what follows, we consider a particular free energy, which contains a quadratic interaction energy, a linear potential and the Boltzmann entropy:
 $$\mathcal{F}(\rho)=\frac{1}{2}\rho^T\mathbb{W}\rho+\mathbb{V}^T\rho+\beta\sum_{i=1}^n\rho_i\log\rho_i\ ,$$
 where $\mathbb{V}\in \mathbb{R}^n$, and $\mathbb{W}\in \mathbb{R}^{n\times n}$ is a symmetric matrix.  
  Its gradient flow becomes
\begin{equation}\label{FPE}
\frac{d\rho_i}{dt}=\sum_{j\in N(i)}\omega_{ij}\theta_{ij}(\mathbb{V}_j-\mathbb{V}_i+(\mathbb{W}\rho)_j-(\mathbb{W}\rho)_i)+\sum_{j\in N(i)}\omega_{ij}\theta_{ij}(\log \rho_j-\log\rho_i)\ , 
\end{equation}
 which is the discrete analog of nonlinear FPE 
 $$\frac{\partial \rho}{\partial t}=\nabla\cdot [\rho\nabla (\mathbb{V}(x)+\int_{\mathbb{R}^d}\mathbb{W}(x,y)\rho(t,y)dy)]+\Delta \rho\ . $$
So we call \eqref{FPE} nonlinear FPE on graphs. A particular attention is given to $$\sum_{j\in N(i)}\omega_{ij}\theta_{ij}(\rho)(\log \rho_j-\log\rho_i)\ ,$$
which can be viewed as a nonlinear representation of Laplacian operator for $\rho$. We shall show that such a nonlinearity is the key for many dynamical properties of \eqref{FPE} later on.

\section{Entropy dissipation}\label{dissipation}
In this section, we focus on the convergence properties of FPE \eqref{FPE}. 
Denote the nonlinear Gibbs measure
$$\rho^\infty_i=\frac{1}{K}e^{-\frac{(\mathbb{W}\cdot \rho^\infty)_i+\mathbb{V}_i}{\beta}}\ ,\quad\textrm{where}\quad K=\sum_{j=1}^n e^{-\frac{(\mathbb{W}\cdot \rho^\infty)_j+\mathbb{V}_j}{\beta}}\ .$$
It is easy to verify that the Gibbs measure is the equilibrium of \eqref{FPE}. Our main theorem here is to show how fast $\rho(t)$, the solution of FPE \eqref{FPE}, converges to $\rho^\infty$.
\begin{theorem}\label{th12}
Assume $\rho^0\in \mathcal{P}_o(G)$ and $\mathcal{F}(\rho)$ is strictly positive definite in $\mathcal{P}(G)$, then
there exists a constant $C>0$,
such that 
\begin{equation}\label{exp}
\mathcal{F}( \rho(t))-\mathcal{F}( \rho^{\infty})\leq e^{-Ct}(\mathcal{F}(\rho^0)-\mathcal{F}( \rho^{\infty}))\ .
\end{equation}
Furthermore
$$C=2m(\rho^0)\lambda_{sec}(\hat{L})\lambda_{\min}(\textrm{Hess}\mathcal{F}) \frac{1}{(r+1)^2}\ ,$$
where
$$r=\sqrt{2} \textrm{Deg}(G)\max_{(i,j)\in E}\omega_{ij} \frac{\|\textrm{Hess}\mathcal{F}\|_1}{\lambda_{\min}(\textrm{Hess}\mathcal{F})^{\frac{3}{2}}}\frac{1-m(\rho^0)}{m(\rho^0)^2}\frac{\lambda_{\max}(\hat{L})}{\lambda_{sec}(\hat{L})^2} \sqrt{\mathcal{F}(\rho^0)-\mathcal{F}(\rho^\infty)}\ ,$$
 $Deg(G)$ is the maximal degree of graph, $\hat{L}=D^TD$ is the graph Laplacian matrix, $\lambda_{sec}(\hat{L})$ and $\lambda_{\max}(\hat{L})$ are the second smallest and the largest eigenvalue of $\hat{L}$ respectively,  
 $$\|\textrm{Hess}\mathcal{F}\|_1=\sup_{\rho\in \mathcal{P}(G)} \|\textrm{Hess}\mathcal{F}(\rho)\|_1\ , \qquad  \lambda_{\min}(\textrm{Hess}\mathcal{F})=\min_{\rho\in \mathcal{P}(G)}  \lambda_{\min}(\textrm{Hess}\mathcal{F}(\rho))\ , $$
 and $$m(\rho_0)=\frac{1}{2}(\frac{1}{1+(2M)^{\frac{1}{\beta}}})^{n-2}\min\{\frac{1}{1+(2M)^{\frac{1}{\beta}}}, \min_{i\in V}{\rho_i^0}\}>0\ ,$$ 
with $M=e^{2\sup_{i\in V, j\in V}|\mathbb{V}_i|+|\mathbb{W}_{ij}|}$.
\end{theorem}

Before giving the complete proof, we want to point out the main difficulties that we must overcome. 
Since $\mathcal{F}(\rho)$ is strictly convex and $\rho^\infty$ is its unique minimizer, it is not hard to show  
$\rho(t)$ converging to $\rho^\infty$. In general, the rate of convergence is determined by comparing the ratio between the first and second derivative of $\mathcal{F}(\rho(t))$ along the gradient flow. If one can find a constant $C>0$, such that
 \begin{equation}\label{C} 
\frac{d^2}{dt^2}\mathcal{F}(\rho(t))\geq -C\frac{d}{dt}\mathcal{F}(\rho(t))\ 
\end{equation}
holds for all $t\geq 0$, 
one can obtain, by integration, 
$$\frac{d}{dt}[\mathcal{F}(\rho^\infty)-\mathcal{F}(\rho(t))]\geq -C[\mathcal{F}(\rho^\infty)-\mathcal{F}(\rho(t))]\ .$$
Then \eqref{exp} is proved following the Gronwall's inequality. 

For FPE \eqref{FPE}, the first derivative of $\mathcal{F}$ along \eqref{FPE} gives
$$ \frac{d}{dt}\mathcal{F}(\rho(t))=F(\rho)^T\dot\rho=-F(\rho)^TL(\rho)F(\rho)=-\dot\rho^TL^{-1}(\rho)\dot\rho\ ,$$
while the second derivative is
\begin{equation*}
\begin{split}
 \frac{d^2}{dt^2}\mathcal{F}(\rho(t))=&2~\dot\rho^T\textrm{Hess}\mathcal{F}(\rho)L(\rho)F(\rho)-F(\rho)^TL(\dot\rho)F(\rho)\\
 =&2~\dot\rho^T\textrm{Hess}\mathcal{F}(\rho)\dot\rho-\dot\rho^TL^{-1}(\rho)L(\dot\rho)L^{-1}(\rho) \dot \rho\ ,\\
 \end{split}
\end{equation*}
where $L(\dot\rho)=D^T\textrm{diag}(\theta_{ij}(\dot \rho))D$. 

Comparing $\frac{d}{dt}\mathcal{F}(\rho(t))$ with $\frac{d^2}{dt^2}\mathcal{F}(\rho(t))$, we find
\begin{equation}\label{EC}
\begin{split}
C:=&\inf_{\rho\in B(\rho^0)}~\frac{2\dot\rho^T\textrm{Hess}\mathcal{F}(\rho)\dot\rho}{\dot\rho^TL^{-1}(\rho)\dot\rho}-\frac{\dot\rho^TL^{-1}(\rho)L(\dot\rho)L^{-1}(\rho) \dot\rho}{\dot\rho^TL^{-1}(\rho)\dot\rho}\ .\\
&\hspace{2cm} \textrm{Quadratic} \hspace{2cm}\textrm{Cubic}
\end{split}
\end{equation}

However, it is not simple to get an estimation of $C$. In the continuous case, there are only a few examples \cite{carrillo2003kinetic}, depending on special interaction potentials $\mathbb{W}$,
that allow us to find $C$ explicitly. In the discrete space, we overcome this difficulty by borrowing techniques from dynamical systems. If $\rho$ is close enough to the equilibrium ($\dot\rho$ is near zero), estimating $C$ in \eqref{EC} becomes possible. This is because the cubic term of $\dot\rho$ in \eqref{EC} becomes one order smaller than $\dot\rho^TL^{-1}(\rho)\dot \rho$, and the dominating quadratic term can be estimated by a solvable eigenvalue 
problem. 

Following this idea, the sketch of proof is as follows: In lemma 6, we first show that the solution of FPE \eqref{FPE} is well defined, and it converges to $\rho^\infty$. In fact, it can be shown $\rho\in B(\rho^0)$, a compact subset in $\mathcal{P}_o(G)$. Then we estimate the convergence rate in $B(\rho^0)$
by two parts, depending on a parameter $x>0$ controlling the closeness between $\rho$ and $\rho^\infty$. If $\rho(t)$ is far away from $\rho^\infty$, the dissipation formula 
$\frac{d}{dt}\mathcal{F}(\rho)=-F(\rho)^TL(\rho)F(\rho)<0$
 gives one convergence rate $r_1(x)$; If $\rho(t)$ is close to $\rho^\infty$, estimating \eqref{EC} is possible. Thus \eqref{C} implies another rate $r_2(x)$. Combing two together, we find a lower bound of dissipation rate $C$ by calculating
$$\max_{x>0}\min\{r_1(x), r_2(x)\}\ .$$

\begin{lemma}
For any initial condition $ \rho^0\in\mathcal{P}_o(G)$, equation \eqref{FPE} has a unique solution $\rho(t): [0,\infty)\rightarrow \mathcal{P}_o(G)$.  Moreover,
\begin{itemize} 
\item[(i)] There exists a constant  $m(\rho_0)>0$, such that $$\rho_i(t)\geq m(\rho_0)>0\ ,$$for all $i\in V$ and $t\geq 0$.
\item [(ii)]$$\lim_{t\rightarrow+\infty}\rho(t)=\rho^\infty\ .$$
\end{itemize}
\end{lemma}
\begin{proof} First, we prove (i) by constructing a compact set $B(\rho^0)\subset \mathcal{P}_o(G)$. 
Denote a sequence of constants $\epsilon_l$, $l=0,1,\cdots, n$, 
\begin{equation*}
\epsilon_1=\frac{1}{2}\min\{\frac{1}{1+(2M)^{\frac{1}{\beta}}}, \min_{i\in V}{\rho_i^0}\} \quad \textrm{and}\quad \epsilon_l=\frac{\epsilon_{l-1}}{1+(2M)^{\frac{1}{\beta}}}, \quad \textrm{for $l=2, \cdots, n$ .}
\end{equation*}
Then we define 
\begin{equation*}
\begin{split}
B(\rho^0)=\{(\rho_i)_{i=1}^n\in \mathcal{P}(G)\mid&\sum_{r=1}^l\rho_{i_r}\leq 1-\epsilon_l ,~\textrm{for any $l\in \{1,\cdots, n-1\}$}, \\ & \textrm{and $1\leq i_1<\cdots<i_l\leq n$} \}.
\end{split}
\end{equation*}
We shall show that if $\rho^0\in B(\rho^0)$, then $\rho(t)\in B(\rho^0)$ for all $t\geq 0$. In other words, the boundary of $B(\rho^0)$ is a repeller for the ODE \eqref{FPE}. 
Assume $\rho(t_1)\in \partial B(\rho^0)$ at time $t_1$, this means that there exist indices $i_1, \cdots, i_l$ with $l\leq n-1$, such that
\begin{equation}\label{assumption}
\sum_{r=1}^l\rho_{i_r}(t_1)=1-\epsilon_l\ .
\end{equation}
We will show \begin{equation*}
\frac{d}{dt}\sum_{r=1}^l\rho_{i_r}(t)|_{t=t_1}<0\ .
\end{equation*}

Let $A=\{i_1, \cdots, i_l\}$ and $A^c=V\setminus A$. On one hand, for any $j\in A^c$,
\begin{equation}\label{1_t1}
\rho_j(t_1)\leq 1-\sum_{r=1}^l \rho_{i_r}(t_1)=\epsilon_l\ .
\end{equation}
On the other hand, since $\rho(t_1)\in B(\rho_0)$, for any $i\in A$, then $\sum_{k\in A\setminus \{i\}}\rho_{k}(t_1)\leq 1-\epsilon_{l-1}$, and from the assumption \eqref{assumption}, $\rho_i(t_1)+\sum_{k\in A\setminus \{i\}} \rho_k(t_1)=1-\epsilon_l$, we obtain
\begin{equation}\label{t2}
\rho_i(t_1)\geq 1-\epsilon_l-(1-\epsilon_{l-1})=\epsilon_{l-1}-\epsilon_l \ .
\end{equation}

Combining equations \eqref{1_t1} and \eqref{t2}, 
we know that for any $i\in A$ and $j\in A^c$,
\begin{equation}\label{t3}
\begin{split}
F_j(\rho)-F_i(\rho)&= (\mathbb{V}_j+(\mathbb{W}\rho)_j)-(\mathbb{V}_i+(\mathbb{W}\rho)_i)+\beta (\log \rho_j-\log \rho_i)\\
&\leq 2\sup_{i,j\in V}|\mathbb{V}_i+\mathbb{W}_{ij}|+\beta (\log \epsilon_l-\log (\epsilon_{l-1}-\epsilon_l))\\
&\leq -\log 2\ ,
\end{split}
\end{equation}
where the last inequality is from $\epsilon_l=\frac{\epsilon_{l-1}}{1+(2M)^{\frac{1}{\beta}}}$ and $M=\sup_{i,j\in V}e^{2(|\mathbb{V}_i|+|\mathbb{W}_{ij}|)}$.

Since the graph is connected, there exists $i_*\in A$, $j_*\in A^c\cap N(i_*)$ such that 
\begin{equation}\label{t4}
\sum_{i\in A}\sum_{j\in A^c\cap N(i^*)}\theta_{ij}(\rho(t_1))\geq \theta_{i_*j_*}(\rho(t_1))>0\ .
\end{equation}
By combining \eqref{t3} and \eqref{t4}, we have 
\begin{equation*}
\begin{split}
\frac{d}{dt}\sum_{r=1}^l\rho_{i_r}(t)|_{t=t_1}=&\sum_{i\in A}\sum_{j\in N(i)}\theta_{ij}(\rho)[F_j(\rho)-F_i(\rho)]|_{\rho=\rho(t_1)}   \\
=&\sum_{i\in A}\{\sum_{j\in A\cap N(i)} \theta_{ij}(\rho)[F_j(\rho)-F_i(\rho)]  \\  
&+~~~\sum_{j\in A^c\cap N(i)} \theta_{ij}(\rho)[F_j(\rho)-F_i(\rho)]   \}|_{\rho=\rho(t_1)}\\
=&\sum_{i\in A}\sum_{j\in A^c\cap N(i)} \theta_{ij}(\rho)[F_j(\rho)-F_i(\rho)]|_{\rho=\rho(t_1)}   \\
\leq &-\log 2\sum_{i\in A}\sum_{j\in A^c\cap N(i)} \theta_{ij}(\rho(t_1))\\
\leq &-\log 2~ \theta_{i_*j_*}(\rho(t_1))<0\ ,
\end{split}
\end{equation*}
where the third equality is from $\sum_{(i,j)\in A}\theta_{ij}(F_j-F_i)=0$. Therefore, we have $\rho(t)\in B(\rho^0)$, thus $\min_{i\in V, t>0}\rho(t)\geq m(\rho^0)$. (ii) can be proved similarly as in \cite{chow2012}, so we omit it here.
\end{proof}

\begin{lemma}\label{lem7}
For $\rho\in \mathcal{P}_o(G)$, then 
$$\lambda_{sec}(\hat{L})\cdot \min_{i\in V}\rho_i\leq   \lambda_{sec}(L(\rho)) \leq  \lambda_{\max}(L(\rho))\leq \max_{i\in V}\rho_i\cdot \lambda_{\max}(\hat{L})\ ,$$
and 
$$\frac{1}{\max_{i\in V}\rho_i\cdot \lambda_{\max}(\hat{L})}  \leq    \lambda_{sec}(L^{-1}(\rho))  \leq  \lambda_{\max}(L^{-1}(\rho))\leq \frac{1}{\min_{i\in V}\rho_i \cdot \lambda_{sec}(\hat{L})}\ .$$
\end{lemma}
\begin{proof}
Since
$$
\min_{i\in V}\rho_i\cdot \frac{1}{2}\sum_{(i,j)\in E}\omega_{ij}(\Phi_i-\Phi_j)^2 \leq \frac{1}{2}\sum_{(i,j)\in E}\omega_{ij}(\Phi_i-\Phi_j)^2\theta_{ij}(\rho)\leq
\max_{i\in V}\rho_i\cdot \frac{1}{2}\sum_{(i,j)\in E}\omega_{ij}(\Phi_i-\Phi_j)^2
\ ,$$
and the Laplacian matrix $\hat{L}$ has the simple eigenvalue $0$ with eigenvector $(1,\cdots, 1)$, then for any vector $\Phi\in\mathbb{R}^n$ with $\sum_{i=1}^n(\Phi_i-\frac{1}{n}\sum_{j=1}^n \Phi_j)^2=1$, we have
$$ \min_{i\in V}\rho_i\cdot \lambda_{sec}(\hat{L})\leq \Phi^TL(\rho)\Phi \leq \max_{i\in V}\rho_i\cdot\lambda_{\max}(\hat{L})\ .$$
This implies that 
$$\min_{i\in V}\rho_i\cdot\lambda_{sec}(\hat{L})\leq   \lambda_{sec}(L(\rho)) \leq  \lambda_{\max}(L(\rho))\leq \max_{i\in V}\rho_i\cdot \lambda_{\max}(\hat{L})\ .$$
By the definition of $L^{-1}(\rho)$, we can prove the other inequality.
\end{proof}

We are now ready to prove the main result.
\begin{proof}[Proof of Theorem \ref{th12}]
 Given a parameter $x>0$, we divide $B(\rho^0)$ into two parts: 
\begin{equation*}
\begin{split}
B(\rho^0)=&\{\rho\in B(\rho^0) ~:~\mathcal{F}(\rho)-\mathcal{F}(\rho^\infty)\geq x \} \cup \{\rho\in B(\rho^0) ~:~\mathcal{F}(\rho)-\mathcal{F}(\rho^\infty)\leq x \}\\
&\hspace{2cm} B_1 \hspace{7cm} B_2
\end{split}
\end{equation*}

We consider the convergence rate in $B_1$ first.
\begin{lemma}\label{LC1} 
Denote $r_1(x)=C_1x$, where $${C_1}=2m(\rho^0)\lambda_{sec}(\hat{L}) \lambda_{\min}(\textrm{Hess}\mathcal{F})\frac{1}{\mathcal{F}(\rho^0)-\mathcal{F}(\rho^\infty)}\ ,$$ 
then
\begin{equation}\label{C1}
\mathcal{F}(\rho(t))-\mathcal{F}(\rho^0)\leq e^{-r_1(x) t}(\mathcal{F}(\rho^0)-\mathcal{F}(\rho^\infty))\ , 
\end{equation}
for any $t\leq T=\inf\{\tau>0\colon \mathcal{F}(\rho)-\mathcal{F}(\rho^\infty)=x\}$.
\end{lemma}

\begin{proof}
We shall show 
$$\frac{\min_{\rho\in B_1}\{ F(\rho)^TL(\rho) F(\rho)\}}{\mathcal{F}(\rho^0)-\mathcal{F}(\rho^\infty)}\geq C_1 x\ .$$
If this is true, then for $t\leq T$,
\begin{equation*}
\begin{split}
\frac{d}{dt}\mathcal{F}(\rho(t))=&-F(\rho)^TL(\rho)F(\rho)\\
\leq &-\min_{\rho\in B_1} F(\rho)^TL(\rho)F(\rho) \frac{\mathcal{F}(\rho(t))-\mathcal{F}(\rho^\infty)}{\mathcal{F}(\rho(t))-\mathcal{F}(\rho^\infty)}          \\
\leq &-\frac{\min_{\rho\in B_1} F(\rho)^TL(\rho)F(\rho)}{\mathcal{F}(\rho^0)-\mathcal{F}(\rho^\infty)}  [\mathcal{F}(\rho(t))-\mathcal{F}(\rho^\infty)]       \\
\leq &-C_1 x   [\mathcal{F}(\rho(t))-\mathcal{F}(\rho^\infty)]     \ .
\end{split}
\end{equation*}
From the Gronwall's inequality, \eqref{C1} is proven. 

By Taylor expansion on $\rho$, we have 
$$\mathcal{F}(\rho^\infty)=\mathcal{F}(\rho)+F(\rho)\cdot (\rho^\infty-\rho)+\frac{1}{2}(\rho^\infty-\rho)^T\textrm{Hess}\mathcal{F}(\bar \rho)(\rho^\infty-\rho)\ ,$$
where $\bar \rho=\rho+s(\rho^\infty-\rho)$, for some constant $s\in (0,1)$. 
Denote the Euclidean projection matrix onto $\mathcal{T}_\rho\mathcal{P}_o(G)$ by 
$$\mathbb{P}=\mathbb{I}-\frac{1}{n}\mathbf{1}\mathbf{1}^T\ ,$$
where $\mathbf{1}=[1,\cdots, 1]^T$ and $\mathbb{I}\in \mathbb{R}^{n\times n}$ is the identity matrix.
Since $\mathbb{P}F(\rho)\cdot(\rho-\rho^\infty)=F(\rho)\cdot (\rho-\rho^\infty)$, then \begin{equation*}
\begin{split}
x\leq \mathcal{F}(\rho)-\mathcal{F}(\rho^\infty)=&\mathbb{P}F(\rho)\cdot (\rho-\rho^\infty)-\frac{1}{2}(\rho-\rho^\infty)^T\textrm{Hess}\mathcal{F}(\bar\rho)(\rho-\rho^\infty)\\
\leq & \|\mathbb{P}F(\rho)\|_2 \|\rho-\rho^\infty\|_2-\frac{1}{2}\lambda_{\min}(\textrm{Hess}\mathcal{F})\|\rho-\rho^\infty\|_2^2\ .\\
\end{split}
\end{equation*}
The above implies 
\begin{equation*}
\begin{split}
\|\mathbb{P}F(\rho)\|_2\geq &\frac{x}{\|\rho-\rho^\infty\|_2}+\frac{1}{2}\lambda_{\min}(\textrm{Hess}\mathcal{F})\|\rho-\rho^\infty\|_2\\
\geq &\sqrt{2x\lambda_{\min}(\textrm{Hess}\mathcal{F})}\ .
\end{split}
\end{equation*}
Thus
\begin{equation*}
\begin{split}
F(\rho)^TL(\rho)F(\rho)=&  \frac{1}{2}\sum_{(i,j)\in E} (F_i(\rho)-F_j(\rho))^2\theta_{ij}(\rho) \\
\geq & \frac{1}{2}\sum_{(i,j)\in E} (F_i(\rho)-F_j(\rho))^2m(\rho^0)  \\
=& \frac{1}{2}\sum_{(i,j)\in E} [\big(F_i(\rho)-\frac{1}{n}\sum_{k=1}^nF_k(\rho)\big)-\big(F_j(\rho)-\frac{1}{n}\sum_{k=1}^n F_k(\rho)\big)]^2m(\rho^0)  \\
= &  m(\rho^0) (\mathbb{P}F(\rho))^T\hat{L}(\mathbb{P}F(\rho))\\
\geq & m(\rho^0)  \lambda_{sec}(\hat{L})\|\mathbb{P}F(\rho)\|_2^2\\
\geq & 2 m(\rho^0)  \lambda_{\min}(\textrm{Hess}\mathcal{F}) \lambda_{sec}(\hat{L}) x\ ,
\end{split}
\end{equation*}
which finishes the proof.
\end{proof}

Next we give the convergence rate in  $B_2$.
\begin{lemma}\label{LC2}
Denote $r_2(x)=C_2-C_3\sqrt{x}$, where
$$C_2=2m(\rho^0)\lambda_{sec}(\hat{L}) \lambda_{\min}(\textrm{Hess}\mathcal{F})\ ,$$
and $$C_3=2\sqrt{2}\textrm{Deg}(G)\max_{(i,j)\in E}\omega_{ij}\frac{\|\textrm{Hess}\mathcal{F}\|_1}{\sqrt{\lambda_{\min}(\textrm{Hess}\mathcal{F})}}\frac{1-m(\rho^0)}{m(\rho^0)} \frac{\lambda_{\max}(\hat{L})}{\lambda_{sec}(\hat{L})}\ .$$
Then 
\begin{equation}\label{C2}
\mathcal{F}(\rho(t))-\mathcal{F}(\rho^0)\leq e^{-r_2(x)(t-T)}(\mathcal{F}(\rho(T))-\mathcal{F}(\rho^\infty))\ , 
\end{equation}
for any $t\geq T=\inf\{\tau>0\colon \mathcal{F}(\rho(\tau))-\mathcal{F}(\rho^\infty)=x\}$.
\end{lemma}

\begin{proof}
We shall show $$\min_{\rho\in B_2}\{ \frac{2\dot\rho^T\textrm{Hess}\mathcal{F}(\rho)\dot\rho-\dot\rho^TL^{-1}(\rho)L(\dot\rho)L^{-1}(\rho) \dot\rho}{\dot\rho^TL^{-1}(\rho)\dot\rho}\}\geq r_2(x)\ .$$
Suppose it is true, then  \begin{equation*} 
\frac{d^2}{dt^2}\mathcal{F}(\rho(t))\geq -r_2(x)\frac{d}{dt}\mathcal{F}(\rho(t))\ 
\end{equation*}
holds for all $t\geq T$. Integrating this formula in $[t,+\infty)$, we obtain
$$\frac{d}{dt}[\mathcal{F}(\rho^\infty)-\mathcal{F}(\rho(t))]\geq -r_2(x)[\mathcal{F}(\rho^\infty)-\mathcal{F}(\rho(t))]\ .$$
By Gronwall's inequality, \eqref{C2} is proven.

We come back to estimate $r_2(x)$. Since $F(\rho^\infty)=c[1,\cdots, 1]^T$ is a constant vector, by Taylor expansion,
we have
\begin{equation*}
\begin{split}
x\geq &\mathcal{F}(\rho)-\mathcal{F}(\rho^\infty) \\
=&F(\rho^\infty)\cdot(\rho-\rho^\infty)+\frac{1}{2}(\rho-\rho^\infty)^T \textrm{Hess}\mathcal{F}(\bar \rho)(\rho-\rho^\infty)\\
 =&\frac{1}{2}(\rho-\rho^\infty)^T \textrm{Hess}\mathcal{F}(\bar \rho)(\rho-\rho^\infty)\\
 \geq &\frac{1}{2} \lambda_{\min}(\textrm{Hess}\mathcal{F}) \|\rho-\rho^\infty\|_2^2\ .
\end{split}
\end{equation*}
Thus 
$$\|\rho-\rho^\infty\|_2\leq \sqrt{\frac{2}{\lambda_{\min}(\textrm{Hess}\mathcal{F})}}\sqrt{x}.$$ Since
\begin{equation*}
\begin{split}
\dot\rho_i&= (L(\rho)F(\rho))_i \\
\leq &\sum_{j\in N(i)}\omega_{ij}|F_i(\rho)-F_j(\rho)|\theta_{ij}(\rho) \\
\leq & \textrm{Deg}(G) \max_{(i,j)\in E}\omega_{ij}  \max_{(i,j)\in E}|F_i(\rho)-F_j(\rho)| \max_{i}\rho_i \\
\leq & \textrm{Deg}(G) \max_{(i,j)\in E}\omega_{ij}  \max_{(i,j)\in E}|F_i(\rho)-F_j(\rho)|  (1-m(\rho^0)) \\
\end{split}
\end{equation*}
and 
\begin{equation*}
\begin{split}
F_i(\rho)-F_j(\rho)=&F_i(\rho^\infty)+\nabla_\rho F_i(\bar \rho) \cdot (\rho-\rho^\infty)-F_j(\rho^\infty)-\nabla_\rho F_j(\tilde \rho)\cdot (\rho-\rho^\infty)\\
=&(\nabla_\rho F_i(\bar\rho)-\nabla_\rho F_j(\tilde \rho))\cdot (\rho-\rho^\infty)\\
\leq&\|\nabla_\rho F_i(\bar\rho)-\nabla_\rho F_j(\tilde \rho)\|_2 \|\rho-\rho^\infty\|_2\\
\leq & 2 \sup_{i\in V, \rho\in \mathcal{P}(G)}\|\nabla_\rho F_i(\rho)\|_2 \|\rho-\rho^\infty\|_2\\
\leq & 2 \sup_{i\in V, \rho\in \mathcal{P}(G)}\|\nabla_\rho F_i(\rho)\|_1 \|\rho-\rho^\infty\|_2\\
= &2 \|\textrm{Hess}\mathcal{F}\|_1 \|\rho-\rho^\infty\|_2\ ,
\end{split}
\end{equation*}
where $\tilde{\rho}$, $\bar \rho$ are two discrete densities between the line segment of $\rho$ and $\rho^\infty$.

Combining these two estimates, we get
$$\|\dot\rho\|_{\infty}\leq  2\cdot\textrm{Deg}(G) \max_{(i,j)\in E}\omega_{ij}  (1-m(\rho^0)) \|\textrm{Hess}\mathcal{F}\|_1 \sqrt{\frac{2x}{\lambda_{\min}(\textrm{Hess}\mathcal{F})}}\ .$$

Denote 
$$L^{-\frac{1}{2}}(\rho)=T\begin{pmatrix}
0 & & &\\
& (\frac{1}{\lambda_{sec}L(\rho)})^{\frac{1}{2}}& &\\
& & \ddots & \\
& & & (\frac{1}{\lambda_{max}L(\rho)})^{\frac{1}{2}}
\end{pmatrix}T^{-1} \ , $$
and $\sigma=\frac{1}{\|L^{-\frac{1}{2}}(\rho)\dot\rho\|_2}L^{-\frac{1}{2}}(\rho)\dot\rho$, thus
\begin{equation*}
\begin{split}
& \frac{2\dot\rho^T\textrm{Hess}\mathcal{F}(\rho)\dot\rho}{\dot\rho^TL^{-1}(\rho)\dot\rho}-\frac{\dot\rho^TL^{-1}(\rho)L(\dot\rho)L^{-1}(\rho) \dot\rho}{\dot\rho^TL^{-1}(\rho)\dot\rho}\\
\geq & 2\frac{\dot\rho^T\textrm{Hess}\mathcal{F}(\rho)\dot\rho}{\dot\rho^TL^{-1}(\rho)\dot\rho}-\|\dot\rho\|_\infty\frac{\dot\rho^TL^{-1}(\rho)\cdot \hat{L}\cdot L^{-1}(\rho) \dot\rho}{\dot\rho^TL^{-1}(\rho)\dot\rho}\\
= &  2\sigma^TL^{\frac{1}{2}}(\rho) \textrm{Hess}\mathcal{F}(\rho)L^{\frac{1}{2}}(\rho)\sigma-a(x)\sigma^TL^{-\frac{1}{2}}(\rho)\hat{L} L^{-\frac{1}{2}}(\rho) \sigma\\
\geq & 2\lambda_{\min}(\textrm{Hess}\mathcal{F}) \sigma^TL(\rho)\sigma -a(x) \lambda_{\max}(\hat{L})\sigma^TL^{-1}(\rho)\sigma\\
\geq & 2\lambda_{\min}(\textrm{Hess}\mathcal{F}) \lambda_{sec}(L(\rho)) -a(x) \lambda_{\max}(\hat{L}) \lambda_{\max}(L^{-1}(\rho))\\
\geq & 2\lambda_{\min}(\textrm{Hess}\mathcal{F}) m(\rho^0)\lambda_{\sec}(\hat{L}) -\frac{a(x) \lambda_{\max}(\hat{L})}{m(\rho^0)\lambda_{sec}(\hat{L})}\\
=& C_2-C_3\sqrt{x}\ ,
\end{split}
\end{equation*}
where the last inequality comes from $\lambda_{sec}(L(\rho))\geq m(\rho^0)\lambda_{sec}(\hat{L})$ and $\lambda_{\max}(L^{-1}(\rho))\leq \frac{1}{m(\rho^0)\lambda_{sec}(\hat{L})}$ in Lemma \ref{lem7}. 
\end{proof}

We are ready to find the overall convergence rate. By Lemma \ref{LC1} and Lemma \ref{LC2}, one can show that for any $t\geq 0$, 
$$\mathcal{F}(\rho(t))-\mathcal{F}(\rho^\infty)\leq e^{-\min\{r_1(x), r_2(x)\}t}(\mathcal{F}(\rho^0)-\mathcal{F}(\rho^\infty))\ ,$$
for any $x>0$. We estimate a constant rate $C$ by showing
 $$\max_{x>0}\min\{C_1 x, C_2-C_3\sqrt{x}\}\geq C\ .$$
It is clear that the maximizer $x^*>0$ is achieved at $C_1 x^*=C_2-C_3\sqrt{x^*}$, i.e. $\sqrt{x^*}= \frac{-C_3+\sqrt{C_3^2+4C_1C_2}}{2C_1}$. Thus 
\begin{equation*}
\begin{split}
C_1x^*= &\frac{(-C_3+\sqrt{C_3^2+4C_1C_2})^2}{4C_1}= \frac{(C_3^2+4C_1C_2-C_3^2)^2}{4C_1(C_3+\sqrt{C_3^2+4C_1C_2})^2}\\
\geq & \frac{16 C_1^2C_2^2}{4C_1\cdot 4 (C_3+\sqrt{C_1C_2})^2}= C_2\frac{1}{ (\frac{C_3}{\sqrt{C_1C_2}}+1)^2}\ ,\\
\end{split}
\end{equation*}
which finishes the proof.
\end{proof}
We remark that for general choice $\theta_{ij}(\rho)\in C^1$, the explicit rate can also be established. Following the proof in Lemma \ref{LC2}, one only needs to replace $r$ by
$\bar r= r\max_{\rho\in B(\rho^0), (i,j)\in E} \frac{\partial \theta_{ij}}{\partial \rho_i}$, then the generalized convergence rate is 
$$C=2m(\rho^0)\lambda_{sec}(\hat{L})\lambda_{\min}(\textrm{Hess}\mathcal{F}) \frac{1}{(\bar r+1)^2}\ .$$
 In addition, when $\mathbb{W}=0$, Theorem \ref{th12} gives the exponential convergence of linear FPE on graphs, for any potential $\mathbb{V}\in \mathbb{R}^n$.

\subsection{Inequalities}
In literature, it is well known that the convergence of FPE can be used to prove the so called Log-Sobolev inequality
and a few others. We mimic this result on graphs and further extend the inequality to the case that includes the nonlinear interaction energy. For simplicity, we take $\beta=1$ and consider $\mathcal{F}(\rho)=\frac{1}{2}\rho^T \mathbb{W}\rho+ \mathbb{V}^T\rho+\sum_{i=1}^n\rho_i\log\rho_i$, which is strictly convex in $\mathcal{P}(G)$. Again, we denote $\rho^\infty$ as the Gibbs measure.

The Log-Sobolev inequality describes a relationship between two functionals named relative entropy and relative Fisher information, which can be expressed using our notations in the following formulas,
\begin{equation}\label{H}
\mathcal{H}(\rho|\rho^\infty):=\mathcal{F}(\rho)-\mathcal{F}(\rho^\infty) \qquad \textrm{Relative entropy} \ ;
\end{equation}
and
\begin{equation}\label{I}
\begin{split}
\mathcal{I}(\rho|\rho^\infty):=&F(\rho)^TL(\rho)F(\rho) \qquad \textrm{Relative Fisher information}\\
=&\frac{1}{2}\sum_{(i,j)\in E}\omega_{ij}(\log\frac{\rho_i}{e^{-(\mathbb{W}\rho)_i-\mathbb{V}_i}}-\log\frac{\rho_j}{e^{-(\mathbb{W}\rho)_j-\mathbb{V}_j}})^2 \theta_{ij}(\rho) \ .
\end{split}
\end{equation}
\begin{corollary}\label{inequality}
If $\mathcal{F}(\rho)$ is strictly convex in $\mathcal{P}(G)$, then there exists a constant $\lambda>0$, such that 
\begin{equation*}
\mathcal{H}(\rho|\rho^\infty)\leq \frac{1}{2\lambda} \mathcal{I}(\rho|\rho^\infty)\ .
\end{equation*}
\end{corollary}
We want to point out that when $\mathbb{W}=0$, corollary \eqref{inequality} is reduced to the standard Log-Sobolev inequality. In this case, functionals \eqref{H} and \eqref{I} can be written as $$\mathcal{H}(\rho)=\sum_{i=1}^n\rho_i\log\frac{\rho_i}{\rho^\infty_i}\ ,\quad \mathcal{I}(\rho)=\frac{1}{2}\sum_{(i,j)\in E}\omega_{ij}(\log\frac{\rho_i}{\rho_i^\infty}-\log\frac{\rho_j}{\rho_j^\infty})^2 \theta_{ij}(\rho)\ .$$
Their continuous counterparts are
$$ \mathcal{H}(\rho)=\int_{\mathbb{R}^d}\rho(x)\log\frac{\rho(x)}{\rho^\infty(x)}dx\ ,\quad \mathcal{I}(\rho)=\int_{\mathbb{R}^d}(\nabla\log\frac{\rho(x)}{\rho^\infty(x)})^2\rho(x)dx\ .$$
\begin{proof}
We use the fact that the dissipation of relative entropy is the relative Fisher information along FPE \eqref{FPE},
\begin{equation*}
\mathcal{I}(\rho(t))=F(\rho)^TL(\rho)F(\rho)=-\frac{d}{dt}\mathcal{H}(\rho(t)|\rho^\infty)\ .
\end{equation*}
Similar as in Theorem \ref{th12}, we divide $\mathcal{P}(G)$ into two regions based on a given parameter $x>0$:
\begin{equation*}
\begin{split}
\mathcal{P}(G)=&\{\rho\in\mathcal{P}(G) ~:~\mathcal{H}(\rho|\rho^\infty)\leq x \} \cup \{\rho\in \mathcal{P}(G)~:~\mathcal{H}(\rho|\rho^\infty)\geq x \}\\
&\hspace{2cm} D_1 \hspace{4.2cm} D_2
\end{split}
\end{equation*}
We shall show two upper bounds of $\frac{\mathcal{H}(\rho)}{\mathcal{I}(\rho)}$ in $D_1$ and $D_2$ respectively.

On one hand, consider FPE \eqref{FPE}, with $\rho(t)$ starting from an initial measure $\rho\in D_1$. Since $\mathcal{H}(\rho(t)|\rho^\infty)$ is a Lyapunov function, then $\rho(t)\in D_1$ for all $t>0$. Following Lemma \ref{LC2}, there exists $r_2(x)>0$, such that 
\begin{equation*}
 \frac{d^2}{dt^2}\mathcal{H}(\rho(t)|\rho^\infty)\geq -r_2(x)\frac{d}{dt}\mathcal{H}(\rho|\rho^\infty)\ ,
 \end{equation*}
which implies
\begin{equation*}
\int_{0}^\infty \frac{d^2}{d\tau^2}\mathcal{H}(\rho(\tau)|\rho^\infty)d\tau\geq \int_{0}^\infty- r_2(x)\frac{d}{d\tau}\mathcal{H}(\rho(\tau)|\rho^\infty)d\tau\ ,
\end{equation*}
i.e.
\begin{equation*}
\mathcal{I}(\rho|\rho^\infty)=\frac{d}{d\tau}\mathcal{H}(\rho(\tau)|\rho^\infty)|^{\tau=\infty}_{\tau=0}\geq r_2(x)\big(\mathcal{H}(\rho(\tau)|\rho^\infty)|^{\tau=0}_{\tau=\infty}\big)=r_2(x)\mathcal{H}(\rho|\rho^\infty)\ ,
\end{equation*}
where $\lim_{t\rightarrow \infty}\frac{d}{dt}\mathcal{H}(\rho(t)|\rho^\infty)=\lim_{t\rightarrow \infty}\mathcal{H}(\rho(t)|\rho^\infty)=0$. 
Thus
 \begin{equation*}
\lambda_1=\sup_{\rho\in D_1}\frac{\mathcal{H}(\rho|\rho^\infty)}{\mathcal{I}(\rho|\rho^\infty)}  \leq \frac{1}{r_2(x)}<\infty\ .
\end{equation*}

On the other hand, if $\rho\in D_2$, we shall show 
\begin{equation*}
\lambda_2=\sup_{\rho\in D_2}\frac{\mathcal{H}(\rho|\rho^\infty)}{\mathcal{I}(\rho|\rho^\infty)}\leq \frac{\sup_{\rho\in D_2}\mathcal{H}(\rho|\rho^\infty)}{\inf_{\rho\in D_2}\mathcal{I}(\rho|\rho^\infty)}<\infty\ .
\end{equation*}
It is trivial that $\mathcal{H}(\rho|\rho^\infty)$ is bounded above. We only need to show $\inf_{\rho\in D_2}\mathcal{I}(\rho|\rho^\infty)>0$.
Assume this is not true, i.e. $\inf_{\rho\in D_2}\mathcal{I}(\rho|\rho^\infty)=0$. Since $\mathcal{I}(\cdot|\rho^\infty)$ is a lower semi continuous function in $\mathcal{P}(G)$,
$\mathcal{I}(\cdot|\rho^\infty)$ is infinity on $\mathcal{P}(G)\setminus \mathcal{P}_o(G)$, and $D_2$ is a compact set, there exists $\rho^*\in D_2\cap \mathcal{P}_o(G)$, such that 
\begin{equation*}
\mathcal{I}(\rho^*|\rho^\infty)=F(\rho^*)^TL(\rho^*)F(\rho^*)\ .
\end{equation*}
This implies $F_i(\rho^*)=F_j(\rho^*)$ for any $(i,j)\in E$. Since $G$ is connected, then $\rho^*=\rho^\infty=\arg\min_{\rho\in \mathcal{P}(G)}\mathcal{F}(\rho)$, which contradicts $\rho^*\in D_2$. By choosing $\frac{1}{2\lambda}=\max\{\lambda_1, \lambda_2\}$, we prove the result. 
\end{proof}
\subsection{Asymptotic properties}
If $\mathbb{W}$ is not a positive definite matrix, there may exist multiple Gibbs measures. Facing these multiple equilibria, it may not be possible to find one explicit rate for any initial conditions, unless there are only a finite
number of equilibria. However, the asymptotic convergence rate can be established whenever the solution is near a equilibrium. In what follows, we study such an asymptotic rate.

Assume that the initial measure $\rho^0$ is in a basin of attraction of an equilibrium $\rho^{\infty}$, meaning \begin{equation*}
(A)\quad \lim_{t\rightarrow \infty} \rho(t)= \rho^{\infty}\quad \textrm{and} \quad 
\textrm{$ \rho^{\infty}$ is an isolated equilibrium}\ .
\end{equation*}

\begin{theorem}\label{asy}
Let (A) hold and 
$$\lambda=\lambda_{sec}(L(\rho^\infty)\cdot \textrm{Hess}\mathcal{F}(\rho^\infty))>0\ .$$ Then   
for any sufficiently small 
$\epsilon>0$ satisfying $(\lambda - \epsilon) >0$,  there exists a time $T>0$, such that when $t>T$,
\begin{equation*}
\mathcal{F}( \rho(t))-\mathcal{F}( \rho^{\infty})\leq e^{-2(\lambda-\epsilon)(t-T)}
(\mathcal{F}( \rho^0)-\mathcal{F}( \rho^{\infty}))\ .
\end{equation*}
\end{theorem}
\begin{proof}
 Since $\lim_{t\rightarrow \infty}\rho(t)=\rho^\infty$, for sufficient small $\epsilon>0$, there exists $t>T$, such that 
$$\lambda_{sec}(L(\rho)\cdot \textrm{Hess}\mathcal{F}(\rho)) \geq \lambda-\frac{1}{2}\epsilon\ , $$
and $$\|\dot\rho\|_\infty=\|L(\rho)F(\rho)\|_\infty\leq \epsilon \frac{m(\rho^0)\cdot\lambda_{sec}(\hat{L})}{\lambda_{\max}(\hat{L})}\ .$$
Similar to the proof of Lemma \ref{LC2},  we have
\begin{equation*}
\begin{split}
\frac{\frac{d^2}{dt^2}\mathcal{F}(\rho)}{\frac{d}{dt}\mathcal{F}(\rho)}= &\frac{2\dot\rho^T\textrm{Hess}\mathcal{F}(\rho)\dot\rho-\dot\rho^TL^{-1}(\rho)L(\dot\rho)L^{-1}(\rho) \dot\rho}{\dot\rho^TL^{-1}(\rho)\dot\rho}\\
\geq &2\lambda_{sec}(L(\rho)\cdot \textrm{Hess}\mathcal{F}(\rho))-\|\dot\rho\|_\infty\cdot \frac{\lambda_{\max}(\hat{L})}{m(\rho^0)\cdot \lambda_{sec}(\hat{L})}\\
\geq & 2(\lambda-\epsilon)\ .
\end{split}
\end{equation*}
Following strategies in \eqref{C}, we prove the result.
\end{proof}

The techniques used in this proof can also be applied to some non-gradient flows, for example, the FPEs with a 
non-symmetric interaction potential $\mathbb{W}$. In this approach, the free energy $\mathcal{F}(\rho)$ is no longer exists. However, the relative Fisher information always exists, which is used to measure the closeness between $\rho(t)$ and $\rho^\infty$.
\begin{corollary}\label{asy2}
Let (A) hold and 
$$\lambda=\lambda_{sec}(L(\rho^\infty)\cdot (\mathcal{J}{F}^T+\mathcal{J}{F})(\rho^\infty))>0\ ,$$
where $\mathcal{J}F$ is the Jacobi operator on vector function $F(\rho)$. Then for any sufficiently small 
$\epsilon>0$ satisfying $(\lambda - \epsilon) >0$,  there exists a time $T>0$, such that when $t>T$,
\begin{equation*}
\mathcal{I}(\rho(t)|\rho^\infty)\leq e^{-2(\lambda-\epsilon)(t-T)}\mathcal{I}(\rho(T)|\rho^\infty)\ .
\end{equation*}
\end{corollary}
\begin{proof}
 Since 
\begin{equation*}
-{\frac{d}{dt}\mathcal{I}(\rho(t)|\rho^\infty)}=\dot\rho^T(\mathcal{J}{F}(\rho)^T+\mathcal{J}{F}(\rho))\dot\rho-\dot\rho^TL^{-1}(\rho)L(\dot\rho)L^{-1}(\rho) \dot\rho\ .
\end{equation*}
Following the proof in Theorem \ref{asy}, it is straightforward to show that if $t>T$, there exists $\epsilon>0$, such that 
$$\frac{d}{dt}\mathcal{I}(\rho(t)|\rho^\infty)\leq -2(\lambda-\epsilon) \mathcal{I}(\rho(t)|\rho^\infty)\ .$$
By the Gronwall's equality, we prove the result.
\end{proof}
In the end, we shall give an explicit formula for the quadratic form in \eqref{EC}, i.e. 
$$\lambda=\min_{\sigma\in \mathcal{T}_\rho\mathcal{P}_o(G)}~\frac{\sigma^T\textrm{Hess}\mathcal{F} \sigma}{\sigma^T L^{-1}(\rho)\sigma}\ . $$
From Lemma \ref{lem2}, there exists a unique $\Phi\in \mathbb{R}^n$, up to constant shift, such that $\sigma=L(\rho)\Phi$. Thus
\begin{equation}\label{LAM}
\begin{split}
\lambda=&\min_{\Phi\in \mathbb{R}^n}~\frac{\Phi^T\cdot L(\rho)\cdot \textrm{Hess}\mathcal{F}\cdot L(\rho)\cdot \Phi}{\Phi^TL(\rho)\Phi}\\
=&\min_{\Phi\in \mathbb{R}^n}~\{\Phi^T\cdot L(\rho)\cdot \textrm{Hess}\mathcal{F}\cdot L(\rho)\cdot \Phi~\colon~ \Phi^TL(\rho)\Phi=1 \}\ .
\end{split}
\end{equation}
We can rewrite the formula \eqref{LAM} explicitly. Introducing
 \begin{equation*}
h_{ij, kl}=(\frac{\partial^2}{\partial \rho_i\partial p_k}+\frac{\partial^2}{\partial \rho_j\partial p_l}-\frac{\partial^2}{\partial \rho_i\partial p_l}-\frac{\partial^2}{\partial \rho_j\partial p_k})\mathcal{F}(p)\ ,\quad \textrm{for any $i$, $j$, $k$, $l\in V$}\ ,
\end{equation*}
we have 
\begin{equation}\label{def}
\lambda=\min_{\Phi\in \mathbb{R}^n} \frac{1}{4}\sum_{(i,j)\in E}
\sum_{(k,l)\in E}\frac{1}{d_{ij}^2d_{kl}^2}h_{ij, kl}(\Phi_i-\Phi_j)\theta_{ij}(\Phi_k-\Phi_l)\theta_{kl}\ ,
\end{equation}
s.t.
\begin{equation*}
 \frac{1}{2}\sum_{(i,j)\in E}\omega_{ij}(\Phi_i-\Phi_j)^2\theta_{ij}=1\ .
\end{equation*}
In fact, it is not hard to show that $\lambda$ is the eigenvalue problem of Hessian operator at the equilibrium in $(\mathcal{P}_o(G), \mathcal{W}_2)$. In the next section, we shall present what \eqref{def} suggests in its continuous 
analog.

\section{Connection with Wasserstein geometry}\label{analog}
We exploit the meaning of $h_{ij,kl}$ by examining its continuous analog in this section. Our calculation
indicates a nice relation to a famous identity in Riemannian geometry, known as the Yano's formula \cite{yano1,
yano2}. 

Consider a smooth finite dimensional Riemannian manifold $\mathcal{M}$. We assume 
that $\mathcal{M}$ is oriented, compact and has no boundary. We denote $\mathcal{P}(\mathcal{M})$ the
space of density functions supported on $\mathcal{M}$, $T_\rho \mathcal{P}(\mathcal{M})$ the tangent space
at $\rho \in \mathcal{P}(\mathcal{M})$, i.e. $T_\rho \mathcal{P}(\mathcal{M})=\{\sigma(x)~:~\int_{\mathcal{M}}\sigma(x)dx=0\}$.
Following Otto calculus in \cite{OV, vil2008}, for any $\sigma(x) \in T_\rho \mathcal{P}(\mathcal{M})$, there exists
a function $\Phi(x)$ satisfying $\sigma(x)=-\nabla\cdot (\rho \nabla \Phi(x))$. This correspondence and the
2-Wasserstein metric endow an scalar inner production on $T_\rho \mathcal{P}(\mathcal{M})$
\begin{equation*}
(\sigma(x), \tilde \sigma (x)) = (\nabla \Phi, \nabla \tilde \Phi)_\rho:=\int_{\mathcal{M}} \nabla \Phi \cdot \nabla \tilde \Phi \rho dx\ .
\end{equation*}


Now consider a smooth free energy $\mathcal{F}:~\mathcal{P}(\mathcal{M})\rightarrow \mathbb{R}$. We assume 
that $\rho^*\in \mathcal{P}(\mathcal{M})$ is an equilibrium satisfying
\begin{equation}\label{localmin}
 \rho^*(x)>0\ ,\quad \nabla \frac{\delta }{\delta \rho(x)} \mathcal{F}(\rho)|_{\rho^*}=0\ ,  
\end{equation}
where $\frac{\delta}{\delta \rho(x)}$ is the first variation operator in $L_2$ metric. 


To understand $h_{ij,kl}$, we calculate the Hessian of $\mathcal{F}$ at $\rho^*$ with respect to the 2-Wasserstein metric, and show
\begin{equation}\label{Thm_Hess}
(\textrm{Hess}_{\mathcal{W}_2}\mathcal{F}\cdot\nabla\Phi, \nabla\Phi)_{\rho^*}= \int_{\mathcal{M}}\int_{\mathcal{M}}(\textrm{D}_{x}\textrm{D}_y\frac{\delta^2}{\delta \rho(x)\delta \rho(y)}\mathcal{F}(\rho)|_{\rho^*}\nabla \Phi(x), \nabla \Phi(y))\rho^*(x)\rho^*(y)dxdy\ ,
\end{equation}
where $\frac{\delta^2}{\delta \rho(x)\delta \rho(y)}\mathcal{F}(\rho)$ is the second variation of functional $\mathcal{F}(\rho)$ in $L_2$ metric,
$\textrm{D}_x$ and $\textrm{D}_y$ are the covariant derivatives in $x$ and $y$ respectively, and $\Phi$ an arbitrary smooth function.

It is known that Hessian can be computed by differentiating the function twice
along the geodesic. Then the Hessian at the equilibrium $\rho^*$ satisfies 
\begin{equation*}
\begin{split}
(\textrm{Hess}_{\mathcal{W}_2}\mathcal{F}\cdot \nabla\Phi, \nabla \Phi)_\rho=&\frac{d^2}{dt^2}\mathcal{F}(\rho_t)|_{t=0}\ .
\end{split}
\end{equation*}
where $\rho_t$ and $\nabla \Phi_t$ are time dependent functions satisfying the geodesic equation \cite{OV, vil2008},
$$\begin{cases}
&\frac{\partial \rho_t}{\partial t}+\nabla\cdot(\rho_t\nabla \Phi_t)=0\\
&\frac{\partial \Phi_t}{\partial t}+\frac{1}{2}(\nabla \Phi_t)^2=0
\end{cases},$$ with an initial measure $\rho|_{t=0}=\rho^*$ and velocity $\nabla\Phi_t|_{t=0}=\nabla\Phi$. 
Following the geodesic, we have 
\begin{equation}\label{Otto_law}
\begin{split}
\frac{d}{dt}\mathcal{F}(\rho_t)=&\int_{\mathcal{M}}  \frac{\delta }{\delta\rho(x)}\mathcal{F}(\rho_t) \frac{\partial \rho_t}{\partial t}dx=-\int_{\mathcal{M}}  \frac{\delta }{\delta\rho(x)}\mathcal{F}(\rho_t)\nabla\cdot (\rho_t\nabla \Phi_t)dx\\
=&\int_{\mathcal{M}} \nabla\frac{\delta }{\delta\rho(x)}\mathcal{F}(\rho_t)\cdot \nabla \Phi_t(x) \rho_t(x)dx\ ,
\end{split}
\end{equation}
where the third equality holds by the integration by parts formula and the fact that $\mathcal{M}$ has no boundary. 

Similarly, we obtain the second order
derivative along the geodesic, 
 \begin{equation*}
\begin{split}
\frac{d^2}{dt^2}\mathcal{F}(\rho_t)=&\int_{\mathcal{M}} \frac{d}{d t}[\nabla \frac{\delta }{\delta\rho(x)}\mathcal{F}(\rho_t)] \cdot\nabla \Phi_t(x)\rho_t(x) dx\hspace{2cm} {(A_1)}\\
&+\int_{\mathcal{M}} \nabla \frac{\delta}{\delta \rho(x)}\mathcal{F}(\rho_t)\cdot \frac{\partial}{\partial t}({\nabla \Phi_t(x)}\rho_t(x)) dx\ .\hspace{1.2cm}{(A_2)}
\end{split}
\end{equation*}
Because $\rho|_{t=0}=\rho^*$ and \eqref{localmin} holds, $(A_2)|_{t=0}=0$. Thus
\begin{equation}\label{Otto_S_1}
\begin{split}
\frac{d^2}{dt^2}\mathcal{F}(\rho_t)|_{t=0}=&(A_1)|_{t=0}=\int_{\mathcal{M}} \frac{d}{dt}[\nabla \frac{\delta }{\delta\rho(x)}\mathcal{F}(\rho_t)]|_{t=0} \cdot\nabla \Phi(x)\rho(x) dx\\
=&\int_{\mathcal{M}} \nabla \frac{d}{d t}\frac{\delta }{\delta\rho(x)}\mathcal{F}(\rho_t)|_{t=0}\cdot\nabla \Phi(x)\rho(x) dx\\
=&-\int_{\mathcal{M}} \frac{d}{d t}\frac{\delta }{\delta\rho(x)}\mathcal{F}(\rho_t)|_{t=0}\nabla\cdot(\nabla \Phi(x)\rho(x)) dx\ .
\end{split}
\end{equation}
In addition, we compute the first order derivative of $\frac{\delta }{\delta\rho(x)}\mathcal{F}(\rho)$ along the 
geodesic and obtain
\begin{equation}\label{Otto_S_2}
\begin{split}
\frac{d}{d t}\frac{\delta }{\delta\rho(x)}\mathcal{F}(\rho_t)|_{t=0}=&\int_{\mathcal{M}}  \frac{\delta^2 }{\delta\rho(x)\delta \rho(y)}\mathcal{F}(\rho) \frac{\partial \rho(t,y)}{\partial t}|_{t=0}dy\\
=&-\int_{\mathcal{M}}  \frac{\delta^2 }{\delta\rho(x)\delta \rho(y)}\mathcal{F}(\rho)\nabla\cdot (\nabla \Phi(y) \rho(y))dy\ .
\end{split}
\end{equation}
 Substituting \eqref{Otto_S_2} into \eqref{Otto_S_1}, we get
 \begin{equation}\label{2}
 \begin{split}
 \frac{d^2}{dt^2}\mathcal{F}(\rho_t)|_{t=0}=&\int_{\mathcal{M}}\int_{\mathcal{M}}\frac{\delta^2}{\delta \rho(x)\delta \rho(y)}\mathcal{F}(\rho)|_{\rho^*}\nabla \cdot (\rho^*(x)\nabla \Phi(x) )\nabla\cdot (\rho^*(y)\nabla \Phi(y))dxdy \\
= &\int_{\mathcal{M}}\int_{\mathcal{M}}(D_{x}D_y\frac{\delta^2}{\delta \rho(x)\delta \rho(y)}\mathcal{F}(\rho)|_{\rho^*}\nabla \Phi(x), \nabla \Phi(y))\rho^*(x)\rho^*(y)dxdy\ , \\
\end{split}
 \end{equation} 
where the second equality is achieved by the integration by parts with respect to $x$ and $y$. Hence, we obtain \eqref{Thm_Hess}. 

Through \eqref{Thm_Hess}, we find the continuous analogue of $h_{ij,kl}$ as   
$$\textrm{D}_x\textrm{D}_y\frac{\delta^2}{\delta \rho(x) \delta \rho(y)}\mathcal{F}(\rho)$$ 

We next illustrate this analog in a particular situation, namely the linear entropy
\begin{equation*}
\mathcal{H}(\rho)=\int_{\mathcal{M}}\rho(x)\log\rho(x)dx\ .
\end{equation*}
In this case, the unique equilibrium (minimizer) $\rho^*(x)=1$ is a uniform measure on $\mathcal{M}$, where the total volume of $\mathcal{M}$ is assumed to be $1$. Hence 
\eqref{2} becomes, 
\begin{equation}\label{yano}
\begin{split}
(\textrm{Hess}_{\mathcal{W}_2}\mathcal{H}\cdot \nabla\Phi,\nabla\Phi)_{\rho^*}=&\int_{\mathcal{M}}[\nabla \cdot (\rho^*\nabla\Phi(x))]^2\frac{1}{\rho^*(x)}dx= \int_{\mathcal{M}}[\nabla\cdot (\nabla \Phi(x))]^2 dx\ .\\
\end{split}
\end{equation} 
The optimal transport theory \cite{OV, vil2008} gives  another formulation of Hessian:
\begin{equation}\label{OV}
(\textrm{Hess}_{\mathcal{W}_2}\mathcal{H}\cdot \nabla\Phi,\nabla\Phi)_{\rho}=\int_{\mathcal{M}} [\textrm{Ric}(\nabla \Phi(x), \nabla\Phi(x))+\textrm{tr}(\textrm{D}^2\Phi(x) \textrm{D}^2\Phi(x)^T)] \rho(x)dx\ ,
\end{equation}
where $\rho$ is an arbitrary density function, $\textrm{D}^2$ is the second covariant derivative and $\textrm{Ric}$ is the Ricci curvature tensor on $\mathcal{M}$. Evaluating \eqref{OV} at the equilibrium $\rho^*$ and comparing it
with \eqref{yano}, we observe
\begin{equation*}
\int_{\mathcal{M}}[\nabla\cdot(\nabla\Phi)]^2dx=\int_{\mathcal{M}} [\textrm{Ric}(\nabla \Phi, \nabla\Phi)+\textrm{tr}(\textrm{D}^2\Phi \textrm{D}^2\Phi^T)] dx\ ,
\end{equation*}
which is the well-known Yano's formula with vector field $\nabla \Phi$. 

\textbf{Acknowledgement:} This paper is based on Wuchen Li's thesis in Chapter 3, which was written under the
supervision of Professor Shui-Nee Chow and Professor Haomin Zhou. We would like to thank Professor Wilfrid Gangbo for many fruitful and inspirational discussions on the related topics.

\end{document}